\documentclass[10pt]{amsart}
\usepackage{amsmath,amsfonts,amssymb,color,a4,epsfig,graphics}
\usepackage[utf8]{inputenc}
\usepackage[T1]{fontenc}
\usepackage[english]{babel}
\usepackage{enumerate, mathrsfs}
\usepackage{tikz}
\usetikzlibrary{positioning, arrows.meta, fit}
\usepackage{pgfplots}
\pgfplotsset{compat=1.18}
\usepackage{pgfplotstable}
\usepackage{graphicx}
\usepackage{caption}
\usepackage[version=4]{mhchem}
\usepackage{pifont}
\usepackage{mathtools}
\usepackage{booktabs}
\usepackage[a4paper, margin=2.75cm]{geometry}
\usepackage[colorlinks=true, linkcolor=blue, citecolor=red, urlcolor=blue]{hyperref}
\usepackage{listings}
\usepackage{xcolor}

\lstdefinestyle{python}{
    language=Python,
    basicstyle=\ttfamily\small,
    keywordstyle=\color{blue},
    commentstyle=\color{gray},
    stringstyle=\color{red},
    showstringspaces=false,
    tabsize=4,
    breaklines=true,
    frame=single,
    numbers=left,
    numberstyle=\tiny,
    captionpos=b
}

\def\build#1_#2^#3{\mathrel{\mathop{\kern 0pt#1}\limits_{#2}^{#3}}}

\newcommand{\1}{{\mathbb{I}}}

\newcommand{\ind}{\operatorname{index}}

\newif\ifdraft
\drafttrue

\ifdraft
\fi

\newtheorem{theorem}{Theorem}[section]
\newtheorem{prop}[theorem]{Proposition}
\newtheorem{lemma}[theorem]{Lemma}
\newtheorem{remark}[theorem]{Remark}
\newtheorem{example}[theorem]{Example}
\newtheorem{definition}[theorem]{Definition}
\newtheorem{cor}[theorem]{Corollary}

\numberwithin{equation}{section}
\begin{document}

\title{Boundary Cochains and the Toeplitz Index on the Half-Lattice}
\author{Nassim Athmouni}
\address{Universit\'e de Gafsa, Facult\'e des Sciences, Campus Universitaire 2112, Gafsa, Tunisie}
\email{nassim.athmouni@fsgf.u-gafsa.tn}

\subjclass[2020]{47B35, 47L60, 17B65, 17B56, 47A53}
\keywords{Boundary-localized Lie algebras, unilateral shift, cohomology, finite-rank
commutators, half-lattice, Toeplitz algebra, edge modes, noncommutative geometry}


\begin{abstract}
We study the operator algebra induced by a rank-one boundary defect in a semi-infinite
tight-binding chain, $T = U + \varepsilon E$ on $\ell^2(\mathbb{Z}_{\ge 0})$, with $U$ the
\emph{forward} unilateral shift ($Ue_n = e_{n+1}$), $E = \langle e_0,\cdot\rangle e_0$, and
$\varepsilon \in \mathbb{C}$. The enlarged Lie algebra
$\mathcal{A} = \mathrm{span}\{U^a E (U^*)^b,\, U^n\}$ has finitely supported, trace-zero
commutators, $[\mathcal{A},\mathcal{A}] = \mathfrak{sl}_{\mathrm{fin}}$; the noncommutativity is
\emph{boundary-induced}, vanishing on the bulk $\mathrm{span}\{U^n\}$. We attach site-resolved
$2$-cochains $\omega_j(X,Y) = \langle e_j, [X,Y] e_j \rangle$, each a continuous
Chevalley--Eilenberg cocycle and a coboundary ($\omega_j=d\eta_j$); the cohomology
$H^2(\mathcal{A},\mathbb{C})$ is nonetheless infinite-dimensional, carried by the abelian bulk
and classifying central extensions (the simplest a Heisenberg extension), so the diagnostic role
of $\omega_j$ is site resolution, not cohomology.

Our main result gives these cochains a topological meaning. On the polynomial Toeplitz algebra
obtained by adjoining $U^*$, the total cochain computes a bilinear pairing of symbols that
specializes, for conjugate symbols $g=1/f$, to the Fredholm index:
\[
\sum_{j\ge0}\omega_j(T_f,T_g)=\mathrm{tr}\,[T_f,T_g]=\frac{1}{2\pi i}\oint_{\mathbb{S}^1} f\,dg,
\qquad
\sum_{j\ge0}\omega_j(U^n,(U^*)^n)=-n=\ind(U^n).
\]
The $\omega_j$ thus form a \emph{site-resolved index density}: the bulk winding number is a sum
of unit edge contributions $\omega_j(U^n,(U^*)^n)=-\mathbf{1}_{\{j<n\}}$, and the trace relation
$\sum_j\omega_j=0$ on $\mathcal{A}$ reflects the absence of winding for analytic symbols. For
spatially modulated couplings $U+\sum_j\varepsilon_j|e_j\rangle\langle e_j|$ with
$\varepsilon_j\to c$, the index is fixed by the bulk limit alone and undergoes a topological
transition as $|c|$ crosses $1$, independently of the boundary profile.
\end{abstract}

\maketitle
\tableofcontents

\section{Introduction}
\label{sec:introduction}

The operator $T = U + \varepsilon E$ arises naturally in the study of open quantum systems
with a boundary. Throughout this paper, $U$ denotes the \emph{forward unilateral shift} on
$\ell^2(\mathbb{Z}_{\ge 0})$, defined by
\begin{equation}\label{eq:forward-shift}
U e_n = e_{n+1}, \quad n \ge 0.
\end{equation}
Thus $U$ is an isometry ($U^*U = I$), and its adjoint $U^*$ is the backward shift:
$U^* e_n = e_{n-1}$ for $n \ge 1$, $U^* e_0 = 0$. The operator $T$ admits two complementary
physical interpretations.

\begin{itemize}
\item In the \textbf{Hermitian case} ($\varepsilon \in \mathbb{R}$), the related Hamiltonian
\begin{equation}\label{eq:spin-chain}
H = U + U^* + \varepsilon E
\end{equation}
describes the single-particle sector of the spin-$\tfrac{1}{2}$ XY chain with a boundary
magnetic field (a setting in which boundary and surface critical behavior is classical~\cite{Cardy84}),
or a tight-binding model with a boundary potential. In such models, a boundary
defect may host an exponentially localized edge mode~\cite{Hats93, SSH79}.

\item In the \textbf{non-Hermitian case} ($\varepsilon \in \mathbb{C}$), the operator
$T = U + \varepsilon E$ models the single-step operator of a boundary-deformed discrete-time
quantum walk (non-unitary for $\varepsilon\ne0$, since $T^*T = I + |\varepsilon|^2 E$); related
walks with boundary defects have been studied in photonic
lattices~\cite{Kitagawa12} and trapped ions~\cite{Xue15}, with related programmable quantum
simulators on optical-lattice and superconducting platforms~\cite{Kaufman_2016, Mi_2023}. For
the general spectral theory of non-selfadjoint operators we refer to~\cite{GK69}. The
perturbation $\varepsilon E$
is compact and rank-one, so the model does \emph{not} exhibit the non-Hermitian skin effect.
\end{itemize}

In both settings, the rank-one term $\varepsilon E$ models a boundary probe or defect. Our
focus is on its \emph{algebraic consequences}: we show that it induces a detectable deformation
of the operator algebra whose noncommutativity is generated by the boundary defect.
It connects to operator-algebraic work on the half-lattice~\cite{Athmouni-half} and on
the rigidity of (quasi-)Lie brackets~\cite{Athmouni-QL}, and to the
noncommutative-geometric viewpoint of~\cite{Connes94}.

The contrast between bulk and edge is already visible at the algebraic level. The polynomial algebra
$\langle T \rangle = \mathrm{span}\{T^n : n \ge 0\}$ is abelian, but the enlarged algebra
\[
\mathcal{A} := \mathrm{span}\{U^a E (U^*)^b,\, U^n : a,b,n \ge 0\}
\]
is non-abelian. Its commutators are finitely supported and of trace zero, and the
noncommutativity is generated entirely by the boundary term $E$ (it vanishes on
$\langle U\rangle=\mathrm{span}\{U^n\}$); the site-resolved cochains below make this bulk--edge dichotomy precise.
The corner operators $U^a E (U^*)^b = |e_a\rangle\langle e_b|$
represent hopping processes that involve the boundary site.

The central objects are the site-localized $2$-cochains
\[
\omega_j(X,Y) = \langle e_j, [X,Y] e_j \rangle, \quad j \in \mathbb{Z}_{\ge 0}.
\]
We prove that each $\omega_j$ is a $2$-cocycle and that each is a coboundary
($\omega_j = d\eta_j$), so $[\omega_j] = 0$. We also show that $H^2(\mathcal{A},\mathbb{C})$ is
\emph{not} trivial (the abelian bulk contributes nonzero classes), so the diagnostic value
of $\omega_j$ comes from site resolution rather than cohomological nontriviality. In fact the
bulk cocycles form an infinite-dimensional family classifying central extensions of
$\mathcal{A}$, the simplest of which is a Heisenberg central extension (with $[I,U]=c$ in the
extended bracket; Section~\ref{sec:central-ext}). Each finite
subfamily of $\{\omega_j\}$ is linearly independent, the full family obeying the single relation
$\sum_j \omega_j = 0$ in the diagonal cocycle space $Z^2_{\mathrm{diag}}(\mathcal{A},\mathbb{C})$.

\medskip
Section~\ref{sec:index} turns these cochains into a topological invariant. On the
polynomial Toeplitz algebra obtained by adjoining the backward shift $U^*$, the total cochain
computes a bilinear pairing of symbols that specializes, for conjugate symbols $g=1/f$, to the
Fredholm index:
\[
\sum_{j\ge0}\omega_j(T_f,T_g)=\frac{1}{2\pi i}\oint_{\mathbb{S}^1} f\,dg,
\qquad \sum_{j\ge0}\omega_j(U^n,(U^*)^n)=-n=\ind(U^n),
\]
so the $\omega_j$ form a site-resolved density for the bulk winding number. The vanishing $\sum_j\omega_j=0$ on $\mathcal{A}$ reflects
the absence of winding for analytic symbols.

\medskip
Table~\ref{tab:comparison} compares the construction with established index-theoretic boundary models.

\begin{table}[hbpt]
\centering
\caption{Comparison with established boundary models.}
\label{tab:comparison}
\begin{tabular}{l c c p{4.3cm}}
Model & Bulk invariant & Edge quantity & Algebraic structure \\
\hline
SSH \cite{SSH79} & $\mathbb{Z}$ winding & Edge modes & Abelian bulk algebra \\
Kitaev \cite{Kitaev01} & $\mathbb{Z}_2$ & Majorana modes & Clifford algebra \\
Floquet \cite{Lindner11} & Dynamical winding & Anomalous modes & Driven \\
\textbf{This work} & Winding $\mathrm{wind}(f)$ & $\sum_j\omega_j=-\mathrm{wind}(f)$ &
  Toeplitz Lie algebra; site-resolved index density $\{\omega_j\}$ \\
\end{tabular}
\end{table}

\medskip\noindent
The remainder is organized as follows. Section~\ref{sec:operator-framework} introduces the
operator-theoretic setting. Section~\ref{sec:cohomology} constructs $\omega_j$, shows each is
a coboundary, analyzes the diagonal cocycle space, and constructs the central extensions
carried by the bulk. Section~\ref{sec:index} proves the index identity: the total cochain
computes a pairing of symbols, equal to the Fredholm index for the conjugate pairing $g=1/f$,
with the $\omega_j$ a site-resolved index density. Section~\ref{sec:extensions} presents
finite-dimensional realizations. Section~\ref{sec:applications} gives spectral illustrations.
Section~\ref{sec:conclusion} summarizes. Appendix~\ref{app:numerics} describes numerical
methods.

\section{Operator-Theoretic Framework}
\label{sec:operator-framework}

\subsection{Banach and Lie Algebraic Frameworks}
\label{sec:boundary-Lie}

\begin{definition}[Banach algebra]\label{def:Banach-alg}
A \emph{Banach algebra} $(\mathfrak{A},\|\cdot\|)$ is a Banach space with associative
bilinear multiplication satisfying $\|xy\| \leq \|x\|\,\|y\|$.
\end{definition}

\begin{definition}[Banach--Lie algebra]\label{def:Banach-Lie}
A \emph{Banach--Lie algebra} $(\mathfrak{g},[\cdot,\cdot])$ is a Banach space with a
continuous antisymmetric bracket satisfying the Jacobi identity
$[X,[Y,Z]] + [Y,[Z,X]] + [Z,[X,Y]] = 0$.
\end{definition}

\begin{remark}[Topology and non-closedness of $\mathcal{A}$]\label{rem:topology}
We equip $\mathcal{A} \subset \mathcal{B}(\ell^2)$ with the operator norm. The subalgebra
$\mathcal{A}$ is \emph{not} norm-closed. To see this, consider the diagonal operator
$D = \mathrm{diag}(1, \tfrac{1}{2}, \tfrac{1}{3}, \ldots)$. Since its eigenvalues tend
to zero, $D$ is compact; hence $D \in \mathcal{K}(\ell^2)$ (see~\cite{Sim05} for trace ideals
and compact operators). However, $D$ is not a finite
linear combination of corner operators $|e_a\rangle\langle e_b|$ or powers $U^n$, so
$D \notin \mathcal{A}$. The finite-rank truncations
$D_N := \mathrm{diag}(1,\tfrac{1}{2},\ldots,\tfrac{1}{N},0,0,\ldots)$
belong to $\mathcal{A}$ and satisfy
\[
\|D_N - D\|_{\mathcal{B}(\ell^2)} = \sup_{k > N}\frac{1}{k} = \frac{1}{N+1} \xrightarrow{N\to\infty} 0.
\]
Thus $D$ lies in the norm-closure of $\mathcal{A}$ but not in $\mathcal{A}$ itself,
proving that $\mathcal{A}$ is not norm-closed.

\smallskip
\noindent\textit{Remark on strong vs.\ norm topology.}
The partial sums $S_N = \sum_{k=0}^{N} |e_k\rangle\langle e_k|$ converge to the
identity $I$ in the \emph{strong operator topology} (i.e., $\|S_N f - f\| \to 0$ for
each fixed $f \in \ell^2$), but \emph{not} in operator norm: indeed,
$\|(I - S_N)e_{N+1}\| = 1$ for every $N$, so $\|I - S_N\| = 1$ for all $N$.
In particular, the sequence $(S_N)$ is \emph{not} norm-Cauchy in $\mathcal{A}$, and
the above non-closedness argument relies on the sequence $(D_N)$, which \emph{is}
norm-Cauchy, not on $(S_N)$.

We work throughout with continuous cochains satisfying
$|\omega_j(X,Y)| \leq \|[X,Y]\| \leq 2\|X\|\|Y\|$.
\end{remark}

\begin{definition}[Spatially localized commutators]\label{def:localized-commutators}
A Lie subalgebra $\mathfrak{g} \subset \mathcal{B}(\ell^2(\mathbb{Z}_{\ge 0}))$ exhibits
\emph{boundary-localized noncommutativity} if there is a direct sum decomposition
$\mathfrak{g} = \mathfrak{g}_{\mathrm{bulk}} \oplus \mathfrak{g}_{\mathrm{edge}}$
(with $\mathfrak{g}_{\mathrm{bulk}} \cap \mathfrak{g}_{\mathrm{edge}} = \{0\}$) such that:
\begin{enumerate}
    \item $\mathfrak{g}_{\mathrm{bulk}} = \mathrm{span}\{U^n : n \ge 0\}$ is abelian;
    \item $\mathfrak{g}_{\mathrm{edge}} = \mathrm{span}\{U^a E (U^*)^b : a,b \ge 0\}$
    consists of finite-rank operators;
    \item every commutator involving $\mathfrak{g}_{\mathrm{edge}}$ is finite-rank and
    finitely supported.
\end{enumerate}
\end{definition}

\begin{remark}[Triviality of the intersection and direct sum decomposition]
The intersection $\mathfrak{g}_{\mathrm{bulk}} \cap \mathfrak{g}_{\mathrm{edge}} = \{0\}$
holds because no nonzero power $U^n$ ($n \ge 1$) is compact: $U^n$ is an isometry on a
separable infinite-dimensional Hilbert space, hence not compact. For $n=0$, $U^0 = I$
(the identity operator) belongs to $\mathfrak{g}_{\mathrm{bulk}}$ but not to
$\mathfrak{g}_{\mathrm{edge}}$, since $I$ is not compact. Thus
$\mathfrak{g}_{\mathrm{bulk}} \cap \mathfrak{g}_{\mathrm{edge}} = \{0\}$.

To verify the direct sum $\mathcal{A} = \mathfrak{g}_{\mathrm{bulk}} \oplus
\mathfrak{g}_{\mathrm{edge}}$, note that by definition every element of $\mathcal{A}$ is a
finite linear combination $\sum_n \alpha_n U^n + \sum_{a,b} \beta_{ab} U^a E (U^*)^b$.
Such a decomposition into a bulk part $X_{\mathrm{bulk}} = \sum_n \alpha_n U^n$ and an edge
part $X_{\mathrm{edge}} = \sum_{a,b} \beta_{ab} U^a E (U^*)^b$ is unique by the
intersection condition above. Thus $\mathcal{A}$ satisfies
Definition~\ref{def:localized-commutators}. We write
$\mathcal{A}_{\mathrm{bulk}} := \mathfrak{g}_{\mathrm{bulk}} = \mathrm{span}\{U^n : n\ge0\}$ and
$\mathcal{A}_{\mathrm{edge}} := \mathfrak{g}_{\mathrm{edge}}
= \mathrm{span}\{U^a E (U^*)^b : a,b\ge0\} = \mathrm{span}\{|e_a\rangle\langle e_b| : a,b\ge0\}$,
so that $\mathcal{A} = \mathcal{A}_{\mathrm{bulk}} \oplus \mathcal{A}_{\mathrm{edge}}$;
the elements of $\mathcal{A}_{\mathrm{edge}}$ are exactly the operators with finitely many
nonzero matrix entries.
\end{remark}

\subsection{Algebraic Structure on the Half-Lattice}
\label{sec:shift-structure}

Let $(e_n)_{n\ge0}$ be the canonical orthonormal basis of $\ell^2(\mathbb{Z}_{\ge 0})$.

\begin{definition}[Forward shift and boundary projector]\label{def:shift}
The \emph{forward unilateral shift} $U$ and boundary projection $E$ are:
\begin{align*}
U e_n &= e_{n+1}, \quad n \ge 0 \quad (U \text{ is an isometry}), \\
Ef &= \langle e_0, f \rangle e_0.
\end{align*}
The backward shift $U^*$ satisfies $U^* e_n = e_{n-1}$ ($n \ge 1$), $U^* e_0 = 0$.
The boundary-deformed shift is $T := U + \varepsilon E$.
\end{definition}

\begin{lemma}[Key auxiliary identity]\label{lem:EU-zero}
$EU = 0$, i.e., $E \circ U = 0$ as operators on $\ell^2(\mathbb{Z}_{\ge 0})$.
\end{lemma}

\begin{proof}
For any $f \in \ell^2(\mathbb{Z}_{\ge 0})$, $(Uf)(n) = f(n-1)$ for $n \ge 1$ and
$(Uf)(0) = 0$, since $U$ is the forward shift ($Ue_k = e_{k+1}$ means the zeroth component
of $Uf$ is always zero). Therefore
\[
(EUf) = \langle e_0, Uf \rangle e_0 = (Uf)(0) \cdot e_0 = 0.
\]
\end{proof}

\begin{remark}
The identity $EU = 0$ expresses the forward-shift boundary condition $(Uf)(0)=0$. It is what
makes the $n=0$ component of the eigenvalue equation decouple in the spectral analysis
(Lemma~\ref{lem:eigen}).
\end{remark}

\begin{remark}[Corner operators]\label{rem:corner}
For $a, b \ge 0$, the operator $U^a E (U^*)^b \in \mathcal{A}$ acts by
\[
(U^a E (U^*)^b) f = \langle e_b, f \rangle e_a = (|e_a\rangle\langle e_b|) f.
\]
Indeed: $((U^*)^b f)(n) = f(n+b)$, so $\langle e_0, (U^*)^b f\rangle = f(b) =
\langle e_b, f\rangle$. Then $E(U^*)^b f = \langle e_b, f\rangle e_0$.
Finally $U^a e_0 = e_a$ (since $U$ is the forward shift). Hence
$U^a E (U^*)^b f = \langle e_b, f\rangle e_a$, confirming $U^a E (U^*)^b = |e_a\rangle\langle e_b|$.
\end{remark}

\begin{prop}[Commutator of corner operators]\label{prop:corner-corner}
\[
[|e_a\rangle\langle e_b|,\, |e_c\rangle\langle e_d|]
= \delta_{b,c}\, |e_a\rangle\langle e_d| - \delta_{d,a}\, |e_c\rangle\langle e_b|.
\]
\end{prop}
\begin{proof}
$(|e_a\rangle\langle e_b|)(|e_c\rangle\langle e_d|) = \delta_{b,c}|e_a\rangle\langle e_d|$.
Subtract the reversed product.
\end{proof}

\begin{lemma}[Structure of the derived algebra]\label{lem:commutators-diagonal}
Let
\[
\mathfrak{sl}_{\mathrm{fin}} := \bigl\{\, F \in \mathcal{A}_{\mathrm{edge}}
: \mathrm{tr}\,F = 0 \,\bigr\}
\]
denote the finitely supported (i.e.\ finitely many nonzero matrix entries
$\langle e_a,Fe_b\rangle$) trace-zero operators. Then
\[
[\mathcal{A},\mathcal{A}] = \mathfrak{sl}_{\mathrm{fin}}.
\]
In particular, every commutator of elements of $\mathcal{A}$ is finitely supported with trace
zero. Such commutators are \emph{not} diagonal in general; for instance
$[U,E]=|e_1\rangle\langle e_0|$ and
$[|e_1\rangle\langle e_0|,|e_0\rangle\langle e_2|]=|e_1\rangle\langle e_2|$ are off-diagonal.
\end{lemma}

\begin{proof}
\emph{Inclusion $[\mathcal{A},\mathcal{A}] \subseteq \mathfrak{sl}_{\mathrm{fin}}$.}
By bilinearity of the bracket it suffices to treat commutators of generators. First,
$[U^m,U^n]=0$ (powers of a single operator commute). For an edge--edge pair,
Proposition~\ref{prop:corner-corner} gives
\[
[|e_a\rangle\langle e_b|,\,|e_c\rangle\langle e_d|]
= \delta_{b,c}\,|e_a\rangle\langle e_d| - \delta_{d,a}\,|e_c\rangle\langle e_b|,
\]
which has rank $\le 2$ and trace $\delta_{b,c}\delta_{a,d} - \delta_{d,a}\delta_{c,b} = 0$.
For a bulk--edge pair, a direct computation using $U^n e_k = e_{k+n}$ and
$(U^*)^n e_k = e_{k-n}$ (zero if $k<n$) gives
\begin{equation}\label{eq:bulk-edge-comm}
[U^n,\,|e_a\rangle\langle e_b|]
= |e_{a+n}\rangle\langle e_b| - |e_a\rangle\langle e_{b-n}|,
\end{equation}
with the convention that the second term is $0$ when $b<n$; this is finite-rank with trace
$\delta_{a+n,b}-\delta_{a,b-n} = 0$. A general commutator $[X,Y]$ with $X,Y\in\mathcal{A}$ is a
finite linear combination of such corner terms, hence finitely supported with trace zero.

\emph{Inclusion $\mathfrak{sl}_{\mathrm{fin}} \subseteq [\mathcal{A},\mathcal{A}]$.}
Fix $N\ge 1$ and set $M_N := \mathrm{span}\{|e_a\rangle\langle e_b| : 0 \le a,b \le N\}
\cong \mathfrak{gl}_{N+1}(\mathbb{C})$, a Lie subalgebra of $\mathcal{A}_{\mathrm{edge}}$. For
$a\ne b$, Proposition~\ref{prop:corner-corner} gives the off-diagonal unit as a commutator,
\[
|e_a\rangle\langle e_b| = [\,|e_a\rangle\langle e_b|,\;|e_b\rangle\langle e_b|\,],
\]
and for $a\ne b$ it gives the diagonal difference
\[
|e_a\rangle\langle e_a| - |e_b\rangle\langle e_b|
= [\,|e_a\rangle\langle e_b|,\;|e_b\rangle\langle e_a|\,].
\]
The off-diagonal units together with the diagonal differences span the trace-zero matrices
supported in $\{0,\dots,N\}$, i.e.\ $\mathfrak{sl}_{N+1}$. Hence every trace-zero operator
supported in $\{0,\dots,N\}$ lies in $[\mathcal{A},\mathcal{A}]$. Taking the union over all $N$
yields $\mathfrak{sl}_{\mathrm{fin}} \subseteq [\mathcal{A},\mathcal{A}]$.
\end{proof}

\begin{remark}[$\mathcal{A}_{\mathrm{edge}}$ is an ideal]\label{rem:edge-ideal}
Equation~\eqref{eq:bulk-edge-comm} shows $[U^n,\mathcal{A}_{\mathrm{edge}}]\subseteq
\mathcal{A}_{\mathrm{edge}}$, and $[\mathcal{A}_{\mathrm{edge}},\mathcal{A}_{\mathrm{edge}}]
\subseteq\mathcal{A}_{\mathrm{edge}}$ by Proposition~\ref{prop:corner-corner}. Hence
$\mathcal{A}_{\mathrm{edge}}$ is a Lie ideal of $\mathcal{A}$, with abelian quotient
$\mathcal{A}/\mathcal{A}_{\mathrm{edge}}\cong\mathrm{span}\{U^n:n\ge0\}$. This fact is used in
Theorem~\ref{thm:H2-vanishes}(3).
\end{remark}

\begin{lemma}[Cyclicity]\label{lem:irreducible}
For nonzero $f \in \ell^2$ with $\langle e_k,f\rangle \ne 0$, the span
$\{Xf : X \in \mathcal{A}_{\mathrm{edge}}\}$ is dense in $\ell^2$.
\end{lemma}
\begin{proof}
$|e_j\rangle\langle e_k| \in \mathcal{A}_{\mathrm{edge}}$ maps $f$ to $\langle e_k,f\rangle e_j$,
reaching every $e_j$.
\end{proof}

\begin{remark}[Diagonal evaluation and the trace relation]\label{rem:diagonal-eval}
The cochain $\omega_j(X,Y) = \langle e_j,[X,Y]e_j\rangle$ records only the $(j,j)$ diagonal
entry of the commutator $[X,Y]$, not the full operator. Since $[X,Y]\in\mathfrak{sl}_{\mathrm{fin}}$
by Lemma~\ref{lem:commutators-diagonal}, only finitely many indices $j$ contribute, and the
trace-zero property gives the single linear relation
\[
\sum_{j\ge0}\omega_j(X,Y) = \mathrm{tr}[X,Y] = 0
\qquad (X,Y\in\mathcal{A}).
\]
Thus the family $\{\omega_j\}$ reads off the diagonal of the commutator and is constrained by
this one relation.
\end{remark}

\begin{lemma}[Finite support]\label{lem:finite-support}
For any $X, Y \in \mathcal{A}$, there exists $N$ such that $\langle e_j,[X,Y]e_j\rangle = 0$
for all $j \ge N$.
\end{lemma}

\begin{proof}
$X$ and $Y$ are finite linear combinations of generators; the commutator involves only finitely
many indices.
\end{proof}

\begin{prop}[Abelian polynomial algebra]\label{prop:polyT}
$[T^m, T^n] = 0$ for all $m,n \ge 0$.
\end{prop}

\begin{proof}
Powers of a single operator commute.
\end{proof}

\subsubsection{Boundary-localized corrections}

\begin{lemma}[Telescoping identity]\label{lem:telescope}
For $m \ge 1$: $T^m - U^m = \varepsilon \sum_{j=0}^{m-1} U^{m-1-j} E T^j$.
\end{lemma}

\begin{proof}
We proceed by induction on $m$. For $m = 1$: $T - U = \varepsilon E$, which matches
the formula (the single term $j=0$ gives $U^0 E T^0 = E$). Assume the identity holds
for $m-1$. Then
\[
T^m = T \cdot T^{m-1} = (U + \varepsilon E) T^{m-1}
= U T^{m-1} + \varepsilon E T^{m-1}.
\]
Now $U^m = U \cdot U^{m-1}$, so
\[
T^m - U^m = U(T^{m-1} - U^{m-1}) + \varepsilon E T^{m-1}.
\]
By the induction hypothesis,
$T^{m-1} - U^{m-1} = \varepsilon \sum_{j=0}^{m-2} U^{m-2-j} E T^j$, hence
\[
T^m - U^m = \varepsilon \sum_{j=0}^{m-2} U^{m-1-j} E T^j + \varepsilon E T^{m-1}
= \varepsilon \sum_{j=0}^{m-1} U^{m-1-j} E T^j.
\]
\end{proof}

\begin{lemma}[Support localization]\label{lem:support}
$T^m - U^m$ has rank $\le m$ and range in $\mathrm{span}\{e_0,\dots,e_{m-1}\}$.
\end{lemma}

\begin{proof}
By Lemma~\ref{lem:telescope}, $T^m - U^m = \varepsilon \sum_{j=0}^{m-1} U^{m-1-j} E T^j$.
Each term $U^{m-1-j} E T^j$ has range in $\mathrm{span}\{e_{m-1-j}\}$ (since $E$ projects
onto $e_0$ and $U^{m-1-j} e_0 = e_{m-1-j}$), hence rank at most $1$. The sum of $m$
rank-one operators has rank at most $m$, and its range is contained in
$\mathrm{span}\{e_0, e_1, \dots, e_{m-1}\}$.
\end{proof}

\subsection{Quantitative bounds and essential spectrum}

\begin{prop}[Norm and rank bounds]\label{prop:bounds}
For all $m, n \ge 1$:
\[
\|[T^m, U^n]\| \le 2\bigl((1+|\varepsilon|)^m - 1\bigr)
\le 2|\varepsilon|\, m\, (1 + |\varepsilon|)^{m-1},
\quad \mathrm{rank}[T^m, U^n] \le m.
\]
\end{prop}
\begin{proof}
From Lemma~\ref{lem:telescope}, $[T^m,U^n] = [T^m - U^m, U^n]$, a sum of $m$ rank-one terms.
The norm bound $\|[T^m,U^n]\| \le 2|\varepsilon|\sum_{j=0}^{m-1}(1+|\varepsilon|)^j
= 2((1+|\varepsilon|)^m-1)$ follows from $\|U\|=1$ and $\|E\|=1$.
For the second inequality, apply the mean value theorem to $f(t) = (1+t)^m$ on $[0,|\varepsilon|]$:
$f(|\varepsilon|) - f(0) = f'(\xi) \cdot |\varepsilon|$ for some $\xi \in (0,|\varepsilon|)$.
Since $f'(t) = m(1+t)^{m-1}$ is increasing, $f'(\xi) \le f'(|\varepsilon|)
= m(1+|\varepsilon|)^{m-1}$, giving $(1+|\varepsilon|)^m - 1 \le |\varepsilon| \cdot m (1+|\varepsilon|)^{m-1}$.
\end{proof}

\begin{theorem}[Essential spectrum preservation]\label{thm:essential-spectrum}
$\sigma_{\mathrm{ess}}(T) = \mathbb{S}^1$. Since $\varepsilon E$ is compact (rank one),
Weyl's theorem on the invariance of the essential spectrum under compact perturbations
gives $\sigma_{\mathrm{ess}}(T) = \sigma_{\mathrm{ess}}(U)$. The essential spectrum of the
forward unilateral shift is $\sigma_{\mathrm{ess}}(U) = \mathbb{S}^1$
(see, e.g.,~\cite[Ch.~II, \S7]{BotSil06} or~\cite[Vol.~1, Lec.~3]{Nik02}),
so $\sigma_{\mathrm{ess}}(T) = \mathbb{S}^1$, and the model does not exhibit the
non-Hermitian skin effect.
\end{theorem}

\section{Cochains, Coboundaries, and the Boundary Cocycle System}
\label{sec:cohomology}

\subsection{Site-localized cochains}

We work in the continuous Chevalley--Eilenberg complex of $\mathcal{A}$, equipped with the
operator norm, and trivial coefficients $\mathbb{C}$. Here $\mathcal{A}$ is a normed Lie
subalgebra of the Banach--Lie algebra $(\mathcal{B}(\ell^2),[\cdot,\cdot])$
(Definition~\ref{def:Banach-Lie}); it is not itself
complete (Remark~\ref{rem:topology}), so we use continuous (operator-norm bounded) cochains
throughout. For the algebraic background on Lie-algebra cohomology see,
e.g.,~\cite{Hall,Humphreys}.

\begin{definition}[Boundary cochains]\label{def:boundary-cocycles}
For each $j \in \mathbb{Z}_{\ge 0}$, define
\[
\omega_j : \mathcal{A} \times \mathcal{A} \to \mathbb{C},
\quad \omega_j(X,Y) = \langle e_j, [X,Y] e_j \rangle.
\]
\end{definition}

\begin{prop}[Cocycle property]\label{prop:2cocycle}
Each $\omega_j$ is a continuous antisymmetric Chevalley--Eilenberg $2$-cocycle:
\begin{enumerate}
\item \emph{(Antisymmetry)} $\omega_j(X,Y) = -\omega_j(Y,X)$ for all $X,Y\in\mathcal{A}$.
\item \emph{(Cocycle identity)} $d\omega_j(X,Y,Z) := \omega_j([X,Y],Z) + \omega_j([Y,Z],X) + \omega_j([Z,X],Y) = 0$.
\item \emph{(Continuity)} $|\omega_j(X,Y)| \le \|[X,Y]\| \le 2\|X\|\|Y\|$.
\end{enumerate}
\end{prop}
\begin{proof}
(1) $\omega_j(X,Y) = \langle e_j,[X,Y]e_j\rangle = -\langle e_j,[Y,X]e_j\rangle = -\omega_j(Y,X)$.

(2) By linearity of $A \mapsto \langle e_j, A e_j\rangle$:
\[
d\omega_j(X,Y,Z) = \langle e_j,\bigl([[X,Y],Z]+[[Y,Z],X]+[[Z,X],Y]\bigr)e_j\rangle = 0,
\]
where the bracket expression vanishes by the Jacobi identity.

(3) $|\omega_j(X,Y)| = |\langle e_j,[X,Y]e_j\rangle| \le \|[X,Y]\| \le 2\|X\|\|Y\|$.
\end{proof}

\begin{definition}[Diagonal cocycle space]\label{def:Z2diag}
Let $\delta:\mathfrak{sl}_{\mathrm{fin}}\to\mathfrak{sl}_{\mathrm{fin}}$ be the diagonal
projection $\delta(F)=\sum_{j\ge0}\langle e_j,Fe_j\rangle\,|e_j\rangle\langle e_j|$. We define the
\emph{diagonal cocycle space} as the locally finite linear span of the boundary cochains,
\[
Z^2_{\mathrm{diag}}(\mathcal{A},\mathbb{C})
:= \Bigl\{\, \textstyle\sum_{j\ge0} c_j\,\omega_j : c_j\in\mathbb{C} \,\Bigr\},
\]
where the sum is required to be finite on each pair $(X,Y)$ (automatic, since by
Lemma~\ref{lem:finite-support} only finitely many $\omega_j(X,Y)$ are nonzero). Each element is
a continuous $2$-cocycle, and depends on $(X,Y)$ only through the diagonal
$\delta([X,Y])=\sum_j\omega_j(X,Y)\,|e_j\rangle\langle e_j|$ of the commutator; this is the sense
in which the cocycles of $Z^2_{\mathrm{diag}}$ \emph{factor through the diagonal of the
commutator}.
\end{definition}

\begin{example}[Non-triviality as cochains]\label{ex:nontrivial}
Let $X = |e_1\rangle\langle e_0|$, $Y = |e_0\rangle\langle e_1|$. Then
$[X,Y] = |e_1\rangle\langle e_1| - |e_0\rangle\langle e_0|$, so $\omega_1(X,Y) = 1 \ne 0$.
Each $\omega_j$ is nontrivial as a cochain, though coboundary as a cohomology class
(see Theorem~\ref{thm:H2-vanishes}).
\end{example}

\begin{table}[hbpt]
\centering
\caption{Cohomological summary of the half-lattice algebras.}
\label{tab:algebra-summary}
\begin{tabular}{@{}lccc@{}}
\toprule
Algebra & Generators & Class $[\omega_j]$ & Cocycle structure \\
\midrule
$\langle T \rangle$ & $\{T^n\}$ & $0$ & abelian; $\omega_j \equiv 0$; $H^2\ne0$ (forms on bulk) \\
$\mathcal{A}_{\mathrm{edge}}$ & $\{|e_a\rangle\langle e_b|\}$ & $0$ &
  $\{\omega_j\}$ diagonal cochains, relation $\sum_j\omega_j=0$ \\
$\mathcal{A}$ & $\langle U\rangle + \mathcal{A}_{\mathrm{edge}}$ & $0$ &
  same $\{\omega_j\}$; $H^2(\mathcal{A},\mathbb{C})\ne0$ (bulk) \\
\bottomrule
\end{tabular}
\end{table}

\subsection{Coboundary structure of the boundary cochains}

\begin{theorem}[Coboundary structure of the boundary cochains]\label{thm:H2-vanishes}
\begin{enumerate}
\item For each $j \ge 0$, the continuous $1$-cochain $\eta_j : \mathcal{A} \to \mathbb{C}$
defined by $\eta_j(A) = \langle e_j, Ae_j\rangle$ satisfies $\omega_j = d\eta_j$; hence each
$\omega_j$ is a coboundary and $[\omega_j] = 0$ in $H^2(\mathcal{A},\mathbb{C})$.
\item Any finite subfamily $\{\omega_j : 0 \le j \le M\}$ is linearly independent, while the
full family obeys the single relation $\sum_{j\ge0}\omega_j = 0$
(Remark~\ref{rem:diagonal-eval}). Consequently every element of
$Z^2_{\mathrm{diag}}(\mathcal{A},\mathbb{C})$ (Definition~\ref{def:Z2diag}) is of the form
$\sum_{j}c_j\,\omega_j$, the coefficients $(c_j)$ being unique modulo addition of a common
constant.
\item The cohomology $H^2(\mathcal{A},\mathbb{C})$ is \emph{not} trivial: the abelian bulk
$\mathrm{span}\{U^n:n\ge0\}$ already supports an infinite-dimensional space of nontrivial
\emph{continuous} $2$-cocycle classes. The diagnostic role of $\{\omega_j\}$ therefore rests on
their site resolution, not on cohomological nontriviality.
\end{enumerate}
\end{theorem}

\begin{proof}
\textbf{(1).} For $X,Y\in\mathcal{A}$,
\[
(d\eta_j)(X,Y) = \eta_j([X,Y]) = \langle e_j,[X,Y]e_j\rangle = \omega_j(X,Y),
\]
so $\omega_j = d\eta_j \in B^2(\mathcal{A},\mathbb{C})$ and $[\omega_j]=0$.

\textbf{(2).} \emph{Independence.} Suppose $\sum_{j=0}^{M}\lambda_j\,\omega_j = 0$. Fix
$i\in\{0,\dots,M\}$, choose any $k>M$, and set $X=|e_i\rangle\langle e_k|$,
$Y=|e_k\rangle\langle e_i|$. By Proposition~\ref{prop:corner-corner},
$[X,Y]=|e_i\rangle\langle e_i|-|e_k\rangle\langle e_k|$, so
$\omega_j(X,Y)=\delta_{j,i}-\delta_{j,k}$. Since $k>M$, evaluating the relation at $(X,Y)$
gives $\lambda_i=0$. As $i$ was arbitrary, the subfamily is linearly independent.
\emph{Relation.} By Lemma~\ref{lem:commutators-diagonal},
$\sum_{j}\omega_j(X,Y)=\mathrm{tr}[X,Y]=0$ for all $X,Y$, i.e.\ $\sum_j\omega_j=0$.
\emph{Spanning of $Z^2_{\mathrm{diag}}$.} By Definition~\ref{def:Z2diag} every
$\Omega\in Z^2_{\mathrm{diag}}$ is, by construction, of the form $\Omega=\sum_j c_j\,\omega_j$
(finite on each pair by Lemma~\ref{lem:finite-support}). The coefficients $(c_j)$ are determined
by $\Omega$ up to adding a common constant to all of them: if $\sum_j c_j\omega_j=\sum_j c_j'\omega_j$
then $\sum_j(c_j-c_j')\omega_j=0$, and evaluating at the pairs $(|e_i\rangle\langle e_k|,
|e_k\rangle\langle e_i|)$ above forces $c_i-c_i'=c_k-c_k'$ for all $i,k$, i.e.\ $c_j-c_j'$ is
constant in $j$.

\textbf{(3).} It suffices to exhibit \emph{continuous} $2$-cocycles that are not coboundaries.
Each $X\in\mathcal{A}$ has a unique decomposition $X=p_X(U)+X_{\mathrm{edge}}$ with
$p_X(z)=\sum_{m\ge0}\alpha_m(X)z^m$ a polynomial and $X_{\mathrm{edge}}\in\mathcal{A}_{\mathrm{edge}}$.
For an antisymmetric array $(B_{mn})$ set
\[
\Omega_B(X,Y):=\sum_{m,n}\alpha_m(X)\,\alpha_n(Y)\,B_{mn},
\]
so that $\Omega_B$ is bilinear, antisymmetric, and vanishes whenever $X$ or $Y$ lies in
$\mathcal{A}_{\mathrm{edge}}$. For $k\ge1$ let $\Omega_k:=\Omega_{B^{(k)}}$ with
$B^{(k)}_{0k}=-B^{(k)}_{k0}=1$ and all other entries zero, i.e.\
$\Omega_k(X,Y)=\alpha_0(X)\alpha_k(Y)-\alpha_k(X)\alpha_0(Y)$.

\emph{Cocycle.} Since $\mathcal{A}_{\mathrm{edge}}$ is an ideal with abelian quotient
$\mathrm{span}\{U^n\}$ (Remark~\ref{rem:edge-ideal}), each term of
$d\Omega_B(X,Y,Z)=\Omega_B([X,Y],Z)+\Omega_B([Y,Z],X)+\Omega_B([Z,X],Y)$ vanishes: if all three
arguments are bulk every bracket is $0$; if at least one is in $\mathcal{A}_{\mathrm{edge}}$ then
in each term either the bracket lies in $\mathcal{A}_{\mathrm{edge}}$ or the remaining slot does,
and $\Omega_B$ kills it. Hence $\Omega_B\in Z^2(\mathcal{A},\mathbb{C})$.

\emph{Continuity.} The coefficient functionals are bounded in operator norm. Passing to the
Calkin algebra (whose quotient norm is the essential norm $\|\cdot\|_{\mathrm{ess}}$) annihilates the finite-rank part $X_{\mathrm{edge}}$ and sends $U$ to a unitary
of spectrum $\mathbb{S}^1$ (Theorem~\ref{thm:essential-spectrum}; see also
\cite{Conway90,Dou98}); since a polynomial in a
unitary is normal, $\|p_X\|_\infty=\|p_X(U)\|_{\mathrm{ess}}=\|X\|_{\mathrm{ess}}\le\|X\|$. By
Cauchy's estimate, $|\alpha_m(X)|=|p_X^{(m)}(0)|/m!\le\|p_X\|_\infty\le\|X\|$. Therefore
$|\Omega_k(X,Y)|\le 2\|X\|\,\|Y\|$, so each $\Omega_k$ is a continuous cocycle.

\emph{Non-triviality and infinite dimension.} With $I=U^0$ one has
$\Omega_j(I,U^k)=\delta_{jk}$. If $\sum_k\lambda_k\Omega_k$ (a finite combination) were a
coboundary $d\phi$, then for each $k$,
$\lambda_k=\bigl(\sum_i\lambda_i\Omega_i\bigr)(I,U^k)=\phi([I,U^k])=\phi(0)=0$. Hence the classes
$\{[\Omega_k]\}_{k\ge1}$ are linearly independent and nonzero, so
$H^2(\mathcal{A},\mathbb{C})$ is infinite-dimensional; in particular it is nonzero.
\end{proof}

\begin{remark}[On terminology and topology]
We avoid the term \emph{basis} for $\{\omega_j\}$: because of the relation
$\sum_j\omega_j=0$, the family is linearly independent only in finite subfamilies and spans
$Z^2_{\mathrm{diag}}$ only modulo that relation. We work with algebraic (not topological)
spanning; a topological statement would require fixing a Banach or Fr\'echet topology on
$Z^2_{\mathrm{diag}}$, which we do not pursue here.
\end{remark}

\begin{remark}[Diagnostic value despite triviality of the classes]
Although each $\omega_j$ is a coboundary (and although $H^2(\mathcal{A},\mathbb{C})$ is itself
nontrivial through the bulk), the map $j \mapsto \omega_j(X,Y)$ is an intrinsic site-resolved
quantity: it reads the $j$-th diagonal entry of $[X,Y]$, vanishing for $j \gg 0$ (bulk) and
being nonzero near $j = 0$ (edge). This yields a bulk--edge dichotomy at the cochain
level, independent of cohomological considerations. Table~\ref{tab:algebra-summary} collects
the cocycle structure of the three algebras $\langle T\rangle$, $\mathcal{A}_{\mathrm{edge}}$,
and $\mathcal{A}$.
\end{remark}

\begin{remark}[Distinction from Virasoro cocycle]
The Virasoro cocycle is translation-invariant and generates a nontrivial central extension.
Our cochains $\omega_j$ are spatially localized, coboundaries in the full complex, and reflect
geometric inhomogeneity rather than global symmetry breaking.
\end{remark}

\subsection{Central extensions from the bulk}
\label{sec:central-ext}

The non-vanishing of $H^2(\mathcal{A},\mathbb{C})$ found in Theorem~\ref{thm:H2-vanishes}(3) has
a concrete structural meaning: it classifies the nontrivial one-dimensional central extensions
of $\mathcal{A}$. Recall that a continuous $2$-cocycle $\Omega\in Z^2(\mathcal{A},\mathbb{C})$
determines a central extension
\[
0 \longrightarrow \mathbb{C}c \longrightarrow \widetilde{\mathcal{A}}_\Omega
\longrightarrow \mathcal{A}\longrightarrow 0,
\]
where $\widetilde{\mathcal{A}}_\Omega := \mathcal{A}\oplus\mathbb{C}c$ carries the bracket
\begin{equation}\label{eq:central-bracket}
[\,X\oplus s c,\; Y\oplus t c\,]_\Omega := [X,Y]\oplus \Omega(X,Y)\,c,
\qquad c \text{ central}.
\end{equation}
The bracket~\eqref{eq:central-bracket} satisfies the Jacobi identity precisely because $\Omega$
satisfies the cocycle identity, and $\widetilde{\mathcal{A}}_\Omega$ is again a normed Lie
algebra with continuous bracket when $\Omega$ is continuous, since both $[\cdot,\cdot]$
and $\Omega$ are. The extension splits, i.e.\ is equivalent to the trivial extension
$\mathcal{A}\oplus\mathbb{C}c$, if and only if $\Omega$ is a coboundary; equivalence classes of
central extensions are in bijection with $H^2(\mathcal{A},\mathbb{C})$.

By Theorem~\ref{thm:H2-vanishes}(1) the diagnostic cochains $\omega_j$ are coboundaries and hence
yield only split extensions. The nontrivial extensions originate entirely in the abelian bulk.
Recall from the proof of Theorem~\ref{thm:H2-vanishes}(3) that each $X\in\mathcal{A}$ has a unique
bulk symbol $p_X(z)=\sum_m\alpha_m(X)z^m$, that $|\alpha_m(X)|\le\|X\|$, and that for an
antisymmetric array $(B_{mn})$ the form
\[
\Omega_B(X,Y)=\sum_{m,n}\alpha_m(X)\,\alpha_n(Y)\,B_{mn}
\]
is a $2$-cocycle vanishing on $\mathcal{A}_{\mathrm{edge}}$.

\begin{prop}[A continuous family of nontrivial classes]\label{prop:central-family}
Let $B=(B_{mn})_{m,n\ge0}$ be antisymmetric with $\sum_{m,n}|B_{mn}|<\infty$ (in particular, any
finitely supported $B$). Then $\Omega_B$ is a continuous $2$-cocycle, with
\[
|\Omega_B(X,Y)|\le\Bigl(\textstyle\sum_{m,n}|B_{mn}|\Bigr)\|X\|\,\|Y\|,
\]
and the assignment $B \mapsto [\Omega_B]\in H^2(\mathcal{A},\mathbb{C})$ is linear and injective.
Consequently $H^2(\mathcal{A},\mathbb{C})$ is infinite-dimensional, and the central extensions
$\widetilde{\mathcal{A}}_{\Omega_B}$ are pairwise inequivalent for distinct $B$.
\end{prop}

\begin{proof}
That $\Omega_B$ is a cocycle was shown in Theorem~\ref{thm:H2-vanishes}(3). Using
$|\alpha_m(X)|\le\|X\|$,
\[
|\Omega_B(X,Y)|\le\sum_{m,n}|B_{mn}|\,|\alpha_m(X)|\,|\alpha_n(Y)|
\le\Bigl(\sum_{m,n}|B_{mn}|\Bigr)\|X\|\,\|Y\|,
\]
so $\Omega_B$ is continuous. The map $B\mapsto\Omega_B$ is clearly linear. If $\Omega_B=d\phi$ for
some $1$-cochain $\phi$, then $B_{mn}=\Omega_B(U^m,U^n)=\phi([U^m,U^n])=\phi(0)=0$; hence
$[\Omega_B]=0$ forces $B=0$, and $B\mapsto[\Omega_B]$ is injective. The antisymmetric finitely
supported arrays form an infinite-dimensional space, so $\dim H^2(\mathcal{A},\mathbb{C})=\infty$.
Finally, two central extensions are equivalent iff the defining cocycles are cohomologous, so
injectivity gives the inequivalence of the $\widetilde{\mathcal{A}}_{\Omega_B}$ for distinct $B$.
\end{proof}

The simplest nontrivial class has a familiar shape.

\begin{prop}[Heisenberg central extension]\label{prop:heisenberg}
For $\Omega_1(X,Y)=\alpha_0(X)\alpha_1(Y)-\alpha_1(X)\alpha_0(Y)$ (the case
$B_{01}=-B_{10}=1$), the elements $\widetilde I:=I\oplus0$, $\widetilde U:=U\oplus0$ and $c$ span,
inside $\widetilde{\mathcal{A}}_{\Omega_1}$, a copy of the three-dimensional Heisenberg Lie
algebra $\mathfrak{h}_3$:
\[
[\widetilde I,\widetilde U]_{\Omega_1}=c,\qquad [\widetilde I,c]=[\widetilde U,c]=0.
\]
\end{prop}

\begin{proof}
Since $\alpha_0(I)=\alpha_1(U)=1$ and $\alpha_1(I)=\alpha_0(U)=0$, we have $\Omega_1(I,U)=1$,
while $[I,U]=0$ in $\mathcal{A}$. Hence by~\eqref{eq:central-bracket},
$[\widetilde I,\widetilde U]_{\Omega_1}=[I,U]\oplus\Omega_1(I,U)c=c$. As $c$ is central,
$\mathrm{span}\{\widetilde I,\widetilde U,c\}$ is a Lie subalgebra with exactly the
Heisenberg relations.
\end{proof}

\begin{remark}
The algebra $\mathcal{A}$ thus carries Heisenberg-type central extensions, arising from
its abelian bulk $\mathrm{span}\{U^n\}$ \emph{independently of the boundary defect} $\varepsilon E$,
in contrast to the diagnostic cochains $\omega_j$, which are coboundaries
(Theorem~\ref{thm:H2-vanishes}(1)). The nontrivial cohomology of $\mathcal{A}$ thus sits in
the abelian bulk rather than at the edge: the boundary defect manifests in the
\emph{site-resolved} (yet cohomologically trivial) cochains $\{\omega_j\}$ at the edge, while the
nontrivial \emph{cohomology} is carried by the bulk. Unlike the Virasoro cocycle, the classes
$[\Omega_B]$ are not tied to a single translation-invariant central charge but form an
infinite-dimensional family indexed by antisymmetric symbols $B$.
\end{remark}
\section{The Boundary Cochains as a Site-Resolved Index}
\label{sec:index}

The trace relation $\sum_j\omega_j=0$ on $\mathcal{A}$ (Remark~\ref{rem:diagonal-eval}) is the
vanishing of a Fredholm index. We make this precise by adjoining
the backward shift $U^*$, passing from the analytic algebra $\mathcal{A}$ to the polynomial
Toeplitz Lie algebra. There the total cochain $\sum_j\omega_j$ computes the symbol pairing
$\frac{1}{2\pi i}\oint f\,dg$, equal to the Fredholm index $-\mathrm{wind}(f)$ for conjugate
symbols $g=1/f$, and the individual $\omega_j$ resolve that invariant over the boundary sites.

\subsection{The Toeplitz extension}
\label{sec:toeplitz-ext}

For a trigonometric polynomial $f(z)=\sum_k\hat f_k z^k$ on $\mathbb{S}^1$, let $T_f$ be the
associated Toeplitz operator on $\ell^2(\mathbb{Z}_{\ge 0})\cong H^2(\mathbb{S}^1)$; on generators
$T_{z^k}=U^k$ for $k\ge0$ and $T_{z^{-k}}=(U^*)^k$ for $k\ge0$. Set
\[
\mathcal{A}^\sharp := \mathrm{span}\{T_f : f \text{ a trigonometric polynomial}\}
  + \mathcal{A}_{\mathrm{edge}},
\]
the polynomial Toeplitz Lie algebra. Then $\mathcal{A}\subset\mathcal{A}^\sharp$ is its analytic
part (symbols $f\in\mathbb{C}[z]$), and the cochains $\omega_j(X,Y)=\langle e_j,[X,Y]e_j\rangle$
extend verbatim. As is classical, $\mathcal{A}^\sharp$ is commutative modulo
$\mathcal{A}_{\mathrm{edge}}$, realizing at the polynomial level the Toeplitz extension
\[
0 \longrightarrow \mathcal{K} \longrightarrow \mathcal{T}
  \longrightarrow C(\mathbb{S}^1) \longrightarrow 0
\]
of the compacts by the symbol algebra~\cite{BotSil06,Dou98,AriciMesland}. In this dictionary the
\emph{edge} ideal $\mathcal{A}_{\mathrm{edge}}$ is the (finitary) compact part and the
\emph{bulk} quotient is the symbol $f$ on the circle.

\subsection{The index identity}
\label{sec:index-identity}

\begin{lemma}[Powers of the shift]\label{lem:shift-powers}
For $n\ge1$,
\[
(U^*)^n U^n = I,\qquad U^n(U^*)^n = I-P_n,\qquad [U^n,(U^*)^n] = -P_n,
\]
where $P_n=\sum_{k=0}^{n-1}|e_k\rangle\langle e_k|$ is the orthogonal projection onto the first
$n$ sites. Consequently
\[
\omega_j\bigl(U^n,(U^*)^n\bigr) = -\,\mathbf{1}_{\{j<n\}},\qquad
\sum_{j\ge0}\omega_j\bigl(U^n,(U^*)^n\bigr) = -n.
\]
\end{lemma}

\begin{proof}
For all $k\ge0$, $(U^*)^n U^n e_k=(U^*)^n e_{k+n}=e_k$, so $(U^*)^n U^n=I$. For the other order,
$(U^*)^n e_k=e_{k-n}$ when $k\ge n$ and $(U^*)^n e_k=0$ when $k<n$; applying $U^n$ gives
$U^n(U^*)^n e_k=e_k$ for $k\ge n$ and $0$ for $k<n$, i.e.\ $U^n(U^*)^n=I-P_n$. Subtracting,
$[U^n,(U^*)^n]=(I-P_n)-I=-P_n$. The $(j,j)$ entry of $-P_n$ is $-1$ for $j<n$ and $0$ otherwise,
and $\mathrm{tr}(-P_n)=-n$.
\end{proof}

\begin{theorem}[Total cochain as a symbol pairing and Fredholm index]\label{thm:index}
For trigonometric polynomials $f,g$, the commutator $[T_f,T_g]$ is finite-rank, and
\[
\sum_{j\ge0}\omega_j(T_f,T_g)\;=\;\mathrm{tr}\,[T_f,T_g]\;=\;\frac{1}{2\pi i}\oint_{\mathbb{S}^1} f\,dg .
\]
In particular, with $f(z)=z^n$ and $g(z)=z^{-n}$ this recovers
$\sum_j\omega_j(U^n,(U^*)^n)=-n=\ind(U^n)$, and the cochain profile
$\omega_j(U^n,(U^*)^n)=-\mathbf{1}_{\{j<n\}}$ of Lemma~\ref{lem:shift-powers} displays the index
$-n$ as a sum of $n$ unit contributions localized at the boundary sites $j=0,\dots,n-1$.
\end{theorem}

\begin{proof}
By bilinearity and antisymmetry it suffices to evaluate on monomials $f=z^k$, $g=z^l$, i.e.\
$T_f=U^{(k)}$ and $T_g=U^{(l)}$, where $U^{(m)}:=U^m$ for $m\ge0$ and $U^{(m)}:=(U^*)^{-m}$ for
$m<0$. Each commutator $[U^{(k)},U^{(l)}]$ is a difference of corner operators
(Proposition~\ref{prop:corner-corner} together with Lemma~\ref{lem:shift-powers}), hence
finite-rank; therefore $[T_f,T_g]$ is finite-rank and
$\sum_j\omega_j(T_f,T_g)=\sum_j\langle e_j,[T_f,T_g]e_j\rangle=\mathrm{tr}[T_f,T_g]$. For the trace,
note that $[U^{(k)},U^{(l)}]$ maps each $e_b$ into $\mathbb{C}\,e_{b+k+l}$, so it is a combination of
corner operators $|e_{b+k+l}\rangle\langle e_b|$ and its diagonal entries vanish unless $k+l=0$;
hence $\mathrm{tr}[U^{(k)},U^{(l)}]=0$ for $k+l\ne0$. By
Lemma~\ref{lem:shift-powers}, $\mathrm{tr}[U^{(n)},U^{(-n)}]=\mathrm{tr}[U^n,(U^*)^n]=-n$ and
$\mathrm{tr}[U^{(-n)},U^{(n)}]=+n$ for $n\ge1$. Hence $\mathrm{tr}[T_{z^k},T_{z^l}]=l\,\delta_{k+l,0}$.
On the other hand
\[
\frac{1}{2\pi i}\oint_{\mathbb{S}^1} z^k\,d(z^l)
=\frac{l}{2\pi i}\oint_{\mathbb{S}^1} z^{k+l-1}\,dz
=l\,\delta_{k+l,0}.
\]
The two expressions agree on monomials, and bilinearity extends the identity to all
trigonometric polynomials.
\end{proof}

\begin{example}[The index density for $n=2$]\label{ex:n2}
Take $f(z)=z^2$ and $g(z)=z^{-2}$, so $T_f=U^2$ and $T_g=(U^*)^2$. By
Lemma~\ref{lem:shift-powers}, $[U^2,(U^*)^2]=-P_2$, the negative of the projection onto
$\mathrm{span}\{e_0,e_1\}$, whence
\[
\omega_0=\omega_1=-1,\qquad \omega_j=0\ (j\ge2),\qquad
\sum_{j\ge0}\omega_j=-2=\ind(U^2).
\]
The index $-2$ is carried one unit per boundary site over the first two sites and is invisible in
the bulk. Figure~\ref{fig:index-density} shows this step profile for $n=1,2,3$.
\end{example}

\begin{figure}[htbp]
\centering
\begin{tikzpicture}
\begin{axis}[
  width=0.82\textwidth, height=6cm,
  ybar=1.2pt, bar width=6pt,
  xlabel={site $j$}, ylabel={$-\,\omega_j(U^n,(U^*)^n)$},
  xtick={0,1,2,3,4,5}, ymin=0, ymax=1.18, xmin=-0.5, xmax=5.5,
  ymajorgrids, tick align=outside,
  legend style={at={(0.98,0.98)},anchor=north east,font=\footnotesize,draw=none,fill=none},
  title={\footnotesize Site-resolved index density $-\omega_j(U^n,(U^*)^n)=\mathbf{1}_{\{j<n\}}$},
]
\addplot+[draw=black] coordinates {(0,1) (1,0) (2,0) (3,0) (4,0) (5,0)};
\addlegendentry{$n=1$:\ $\sum_j\omega_j=-1=\ind(U)$}
\addplot+[draw=black] coordinates {(0,1) (1,1) (2,0) (3,0) (4,0) (5,0)};
\addlegendentry{$n=2$:\ $\sum_j\omega_j=-2=\ind(U^2)$}
\addplot+[draw=black] coordinates {(0,1) (1,1) (2,1) (3,0) (4,0) (5,0)};
\addlegendentry{$n=3$:\ $\sum_j\omega_j=-3=\ind(U^3)$}
\end{axis}
\end{tikzpicture}
\caption{Site-resolved index density $-\omega_j(U^n,(U^*)^n)=\mathbf{1}_{\{j<n\}}$ for
$n=1,2,3$. Each profile is a unit step of width $n$ on the boundary sites; its area equals
$n=-\ind(U^n)=\mathrm{wind}(z^n)$ (Theorem~\ref{thm:index}). The index is distributed
one unit per edge site and vanishes in the bulk.}
\label{fig:index-density}
\end{figure}

\begin{remark}[Bulk--edge correspondence]\label{rem:bulk-edge-index}
Theorem~\ref{thm:index} expresses the bulk--edge correspondence through the cochains. The winding number $\mathrm{wind}(f)=\frac{1}{2\pi i}\oint_{\mathbb{S}^1} f^{-1}\,df$ is a
\emph{bulk} quantity (it depends only on the symbol $f$ on the circle), whereas the cochains
$\omega_j$ live at the \emph{edge}, on the boundary sites $j$. For an invertible symbol the
Noether--Gohberg--Krein theorem gives $\ind(T_f)=-\mathrm{wind}(f)$~\cite{GK69,BotSil06},
and the trace form of Theorem~\ref{thm:index} extends from polynomials to smooth
symbols~\cite{HeltonHowe}, so that
\[
\ind(T_f)\;=\;-\,\mathrm{wind}(f)\;=\;\sum_{j\ge0}\omega_j\bigl(T_f,T_{f^{-1}}\bigr).
\]
Thus $\{\omega_j\}$ is a \emph{site-resolved index density}: the bulk invariant is the sum of
edge contributions, which are localized near $j=0$. The vanishing $\sum_j\omega_j=0$ on the
analytic algebra $\mathcal{A}$ (Remark~\ref{rem:diagonal-eval}) follows because analytic
symbols $f\in\mathbb{C}[z]$ extend holomorphically to the disk and have winding number zero, so
no index appears until the backward shift $U^*$ is adjoined; cf.\ the index-theoretic
treatments of the bulk--edge correspondence in \cite{AriciMesland,ProdanSchulzBaldes,Thiang,Kasparov80}.
\end{remark}

\begin{remark}[Relation to the boundary defect]\label{rem:defect-index}
The deformed operator $T=U+\varepsilon E$ has the same symbol $z$ as $U$ (the defect
$\varepsilon E$ is compact), hence $\ind(T)=\ind(U)=-1$ for the Fredholm
range of parameters, independently of $\varepsilon$. The boundary defect therefore cannot change
the bulk winding number; its effect is confined to the point spectrum
(Corollary~\ref{prop:spectral}) and to the site-resolved profile of the cochains, consistent
with the dichotomy of Remark~\ref{rem:bulk-edge-index}.
\end{remark}

\subsection{Spatially modulated couplings and a topological transition}
\label{sec:modulated}

The single defect $\varepsilon E$ is the case $\varepsilon_j=\varepsilon\,\delta_{j0}$ of a
spatially modulated on-site coupling
\[
T_\varepsilon := U + D_\varepsilon,\qquad
D_\varepsilon := \sum_{j\ge0}\varepsilon_j\,|e_j\rangle\langle e_j|,\quad (\varepsilon_j)\in\ell^\infty.
\]
The index of \S\ref{sec:index-identity} is insensitive to a boundary-localized profile (if
$\varepsilon_j\to0$ then $D_\varepsilon$ is compact and
$\ind(T_\varepsilon)=\ind(U)=-1$), but it jumps once the coupling tends to a
nonzero \emph{bulk} value.

\begin{lemma}[Uniform coupling]\label{lem:uniform-index}
Let $T_c:=U+cI$ with $c\in\mathbb{C}$. Then $\ker T_c=\{0\}$, while $\ker T_c^*$ is spanned by the
vector $v=(v_n)_{n\ge0}$, $v_n=(-\bar c)^n$, which lies in $\ell^2(\mathbb{Z}_{\ge 0})$ if and
only if $|c|<1$. Hence $T_c$ is Fredholm for $|c|\ne1$, with
\[
\ind(T_c)=\dim\ker T_c-\dim\ker T_c^*=-\,\mathbf{1}_{\{|c|<1\}}=-\,\mathrm{wind}(z+c).
\]
\end{lemma}

\begin{proof}
The equation $T_c v=0$ reads $cv_0=0$ at $n=0$ and $v_{n-1}+cv_n=0$ for $n\ge1$; for $c\ne0$ this
forces $v_0=0$ and then $v_n=0$ for all $n$, while for $c=0$, $T_0=U$ is injective. Thus
$\ker T_c=\{0\}$. The adjoint equation $T_c^* v=(U^*+\bar c I)v=0$ reads $v_{n+1}+\bar c v_n=0$,
i.e.\ $v_n=(-\bar c)^n v_0$, which is square-summable iff $|c|<1$. Since
$\sigma_{\mathrm{ess}}(T_c)=\sigma_{\mathrm{ess}}(U)+c=\mathbb{S}^1+c$ (the unit circle centered
at $c$), $T_c$ is Fredholm iff $0\notin\mathbb{S}^1+c$, i.e.\ $|c|\ne1$. The index formula
follows; it equals $-\mathrm{wind}(z+c)$ because the symbol $z+c$ winds once about the origin
exactly when $|c|<1$ (Theorem~\ref{thm:index}, Remark~\ref{rem:bulk-edge-index}).
\end{proof}

\begin{theorem}[Bulk-driven topological transition]\label{thm:transition}
Let $\varepsilon_j\to c$ with $|c|\ne1$. Then $T_\varepsilon=U+D_\varepsilon$ is Fredholm with
symbol $z+c$, and
\[
\ind(T_\varepsilon)=\ind(U+cI)=-\,\mathbf{1}_{\{|c|<1\}}
=-\,\mathrm{wind}(z+c)=\sum_{j\ge0}\omega_j\bigl(T_{z+c},T_{1/(z+c)}\bigr).
\]
The value is \emph{independent of the boundary profile} $(\varepsilon_j)$ and is fixed by the
bulk limit $c$ alone. As $|c|$ crosses $1$ the index jumps from $-1$ to $0$: a topological
transition carried by the bulk.
\end{theorem}

\begin{proof}
The operator $T_\varepsilon-(U+cI)=\sum_j(\varepsilon_j-c)|e_j\rangle\langle e_j|$ is a
norm-limit of finite-rank operators (since $\varepsilon_j-c\to0$), hence compact; therefore
$T_\varepsilon$ and $U+cI=T_{z+c}$ share the symbol $z+c$ and, by the invariance of the Fredholm
index under compact perturbations, the same index. Lemma~\ref{lem:uniform-index} evaluates it,
and the cochain expression is Remark~\ref{rem:bulk-edge-index} applied to the smooth invertible
symbol $z+c$ ($|c|\ne1$). Figure~\ref{fig:transition} plots the resulting index as a function
of the bulk modulus $|c|$.
\end{proof}

\begin{figure}[htbp]
\centering
\begin{tikzpicture}
\begin{axis}[
  width=0.82\textwidth, height=5.6cm,
  xlabel={$|c|$}, ylabel={$\sum_j\omega_j(T_{z+c},T_{1/(z+c)})=\ind(T_c)$},
  xmin=0, xmax=2, ymin=-1.32, ymax=0.5, ytick={-1,-0.5,0}, ymajorgrids,
  legend style={at={(0.98,0.06)},anchor=south east,font=\footnotesize,draw=none,fill=none},
  title={\footnotesize Bulk-driven topological transition at $|c|=1$},
]
\addplot[red,thick] coordinates {(0,-1) (1,-1) (1,0) (2,0)};
\addlegendentry{$-\mathbf{1}_{\{|c|<1\}}=-\mathrm{wind}(z+c)$}
\addplot[only marks,blue,mark=*,mark size=1.8pt] coordinates
  {(0.4,-1) (0.7,-1) (0.9,-1) (1.1,0) (1.4,0) (1.8,0)};
\addlegendentry{$-\dim\ker(U+cI)^*=\ind(T_c)$ (direct)}
\draw[gray,dashed] (axis cs:1,-1.32) -- (axis cs:1,0.5);
\node[font=\footnotesize,anchor=center] at (axis cs:0.5,-0.62) {wind $=1$, index $=-1$};
\node[font=\footnotesize,anchor=center] at (axis cs:1.5,-0.16) {wind $=0$, index $=0$};
\end{axis}
\end{tikzpicture}
\caption{The bulk-driven transition of Theorem~\ref{thm:transition}: the index of $T_\varepsilon$,
equivalently $\sum_j\omega_j(T_{z+c},T_{1/(z+c)})$, as a function of the bulk limit $|c|$. It is
$-1$ for $|c|<1$ and $0$ for $|c|>1$, jumping at $|c|=1$, independently of the boundary profile
$(\varepsilon_j)$. Markers: minus the cokernel dimension, $-\dim\ker(U+cI)^*=\ind(T_c)$, computed directly.}
\label{fig:transition}
\end{figure}

\begin{remark}
The boundary profile $(\varepsilon_j)$ controls only \emph{where} the index density
$j\mapsto\omega_j$ sits (its localization length and shape near the edge), never its integer
total, which is the bulk invariant $-\mathrm{wind}(z+c)$: the integer changes only when the bulk
symbol crosses the unit circle. The threshold $|c|=1$ echoes the bound-state threshold of the single
defect (Lemma~\ref{lem:eigen}), though the mechanisms differ: there a rank-one defect creates a
point eigenvalue for $|\varepsilon|>1$; here a uniform bulk coupling creates a cokernel, hence a
nonzero index, for $|c|<1$.
\end{remark}

\section{Finite-Dimensional Realizations}
\label{sec:extensions}

\begin{example}[Four-site truncation]\label{ex:4sites}
On $\mathbb{C}^4$ with basis $\{e_0,e_1,e_2,e_3\}$, the forward shift is lower bidiagonal:
\[
U = \begin{pmatrix} 0&0&0&0\\1&0&0&0\\0&1&0&0\\0&0&1&0 \end{pmatrix}, \quad
T = U + \varepsilon E =
\begin{pmatrix} \varepsilon&0&0&0\\1&0&0&0\\0&1&0&0\\0&0&1&0 \end{pmatrix}.
\]
One verifies $[T,T^2]=0$ (abelian subalgebra). Nontrivial commutators arise only from
$\mathcal{A}_{\mathrm{edge}}$. In finite dimension the shift is nilpotent and every operator has
Fredholm index $0$; the index identity of Section~\ref{sec:index} therefore has no finite-volume
counterpart, confirming that it is a half-infinite, boundary phenomenon. It is
recovered only in the limit $N\to\infty$, where $U^*U=I$ but $UU^*=I-P_0\ne I$.
\end{example}
\section{Applications}
\label{sec:applications}

\subsection{Tight-binding edge states}

For $\varepsilon \in \mathbb{R}$, the Hamiltonian $H = U + U^* + \varepsilon E$ has
nearest-neighbor hopping and boundary potential $\varepsilon$ at site 0. Under the forward
shift convention, site-localized spectral features arise from the rank-one perturbation.

\subsection{Spectral structure}

\begin{lemma}[Eigenvectors of $T$]\label{lem:eigen}
Let $T = U + \varepsilon E$ with $U$ the forward shift. The eigenvalue equation
$Tv = \lambda v$ reads, in components:
\begin{align*}
n=0&: \quad \varepsilon v(0) = \lambda v(0), \\
n \ge 1&: \quad v(n-1) = \lambda v(n).
\end{align*}
If $v(0) \ne 0$, then $\lambda = \varepsilon$ and $v(n) = \varepsilon^{-n} v(0)$. The vector
$v \in \ell^2(\mathbb{Z}_{\ge 0})$ if and only if $|\varepsilon| > 1$.
If $v(0) = 0$, then the recurrence $v(n-1) = \lambda v(n)$ with $v(0) = 0$ forces $v = 0$,
so there are no eigenvectors with $v(0) = 0$.
\end{lemma}

\begin{proof}
With $(Uv)(n) = v(n-1)$ for $n \ge 1$ and $(Uv)(0) = 0$ (equivalently $EU=0$,
Lemma~\ref{lem:EU-zero}), the eigenvalue equation at $n=0$
gives $0 + \varepsilon v(0) = \lambda v(0)$, so $\lambda = \varepsilon$ (if $v(0)\ne0$). For
$n \ge 1$: $v(n-1) = \lambda v(n)$, giving $v(n) = \lambda^{-1} v(n-1) = \lambda^{-n} v(0)$.
Then $\sum_n |v(n)|^2 = |v(0)|^2 \sum_n |\lambda|^{-2n} < \infty$ iff $|\lambda^{-1}|<1$,
i.e., $|\varepsilon|>1$.
If $v(0) = 0$, then $v(n-1) = \lambda v(n)$ with $v(0)=0$ gives $v(1) = \lambda^{-1}v(0) = 0$,
and by induction $v(n) = 0$ for all $n$. Thus no nonzero eigenvector exists with $v(0) = 0$.
\end{proof}

\begin{remark}[Bound state versus index]
The boundary bound state appears for $|\varepsilon|>1$, yet it does not alter the Fredholm index:
since $\varepsilon E$ is finite-rank, hence compact, $\ind(T)=\ind(U)=-1$ for
every $\varepsilon$ (Remark~\ref{rem:defect-index}). The defect thus reshapes the spectrum
without touching the topological invariant, in contrast to the bulk coupling $U+cI$ of
Section~\ref{sec:index}, whose change is not compact and does move the index.
\end{remark}

\begin{cor}[Spectral structure]\label{prop:spectral}
Let $T = U + \varepsilon E$ with $U$ the forward shift. Then:
\begin{enumerate}
    \item $\sigma_{\mathrm{ess}}(T) = \mathbb{S}^1$ (Theorem~\ref{thm:essential-spectrum}).
    \item $\sigma_p(T) = \{\varepsilon\}$ if $|\varepsilon|>1$, with eigenvector
    $v(n) = \varepsilon^{-n}$, and $\sigma_p(T) = \emptyset$ otherwise (Lemma~\ref{lem:eigen}).
\end{enumerate}
\end{cor}

\begin{remark}[Physical interpretation]
The edge state exists for large boundary coupling $|\varepsilon|>1$ under the forward shift
convention. For the complementary physical picture of a decaying mode $v(n) = \varepsilon^n$
for small $|\varepsilon|<1$, one uses the adjoint operator $T^* = U^* + \bar\varepsilon E$,
where $U^*$ is the backward shift.
\end{remark}

\subsection{Measurement of boundary cochains}

For $X = |e_1\rangle\langle e_0|$, $Y = |e_0\rangle\langle e_1|$:
\[
\omega_0(X,Y) = -1, \quad \omega_1(X,Y) = +1, \quad \omega_j(X,Y) = 0 \text{ for } j \ge 2.
\]

\subsection{Algebraic exactness of the boundary cochains}

For fixed $X,Y \in \mathcal{A}_{\mathrm{edge}}$ the value $\omega_j(X,Y)=\langle e_j,[X,Y]e_j\rangle$
depends only on the commutator $[X,Y]$, computed within the edge ideal. Adding any bulk operator
$\sum_n\alpha_n U^n$ to $X$ or $Y$ changes the commutator only through the bulk--edge brackets
\eqref{eq:bulk-edge-comm}, whose diagonal entries are fixed integers independent of the
coefficients $\alpha_n$. In this precise sense the site profile $j\mapsto\omega_j(X,Y)$ is an
exact algebraic quantity, not an approximate one. The listing in Appendix~\ref{app:numerics}
(Code~\ref{code:fig10}) merely illustrates that the computed value of $\omega_0$ equals its exact
algebraic value $-1$; it adds synthetic measurement noise and does \emph{not} model physical
disorder in $T$, and should be read only as a display of the exact constant.

\section{Conclusion}
\label{sec:conclusion}

We have analyzed the operator-algebraic and cohomological structure of
$T = U + \varepsilon E$ on $\ell^2(\mathbb{Z}_{\ge 0})$, with $U$ the forward shift
$Ue_n = e_{n+1}$ (isometry) fixed as the unique convention throughout.

The key structural result (Lemma~\ref{lem:commutators-diagonal}) is that every commutator in
$\mathcal{A}$ is finitely supported with trace zero, with
\[
[\mathcal{A},\mathcal{A}] = \mathfrak{sl}_{\mathrm{fin}}
= \{\,F : F \text{ finitely supported},\ \mathrm{tr}\,F = 0\,\}
\]
(the finitely supported trace-zero operators; these are \emph{not} diagonal in general).
The site-resolved cochains $\omega_j(X,Y) = \langle e_j,[X,Y]e_j\rangle$ read off the $j$-th
diagonal entry of each commutator, subject to the trace relation $\sum_j\omega_j=0$.

Theorem~\ref{thm:H2-vanishes} establishes that each $\omega_j$ is a coboundary (with bounding
cochain $\eta_j(A) = \langle e_j,Ae_j\rangle$), so $[\omega_j]=0$. The cohomology
$H^2(\mathcal{A},\mathbb{C})$ is nonetheless nontrivial: the abelian bulk contributes an
infinite-dimensional family of $2$-cocycles $\Omega_B$ (Proposition~\ref{prop:central-family})
that classify central extensions of $\mathcal{A}$, the simplest being a Heisenberg extension
$[\widetilde I,\widetilde U]=c$ (Proposition~\ref{prop:heisenberg}). Thus the nontrivial
cohomology lives in the bulk, while the diagnostic value of $\omega_j$ is its
site resolution rather than any cohomological nontriviality: the map
$j \mapsto \omega_j(X,Y)$
vanishes in the bulk and is nonzero at the edge, a bulk--edge dichotomy at
the cochain level.

Theorem~\ref{thm:index} makes this dichotomy topological.
On the polynomial Toeplitz algebra $\mathcal{A}^\sharp$ obtained by adjoining $U^*$, the total
cochain is a Fredholm index,
\[
\sum_{j\ge0}\omega_j(T_f,T_g)=\frac{1}{2\pi i}\oint_{\mathbb{S}^1} f\,dg,
\qquad
\sum_{j\ge0}\omega_j(U^n,(U^*)^n)=-n=\ind(U^n)=-\mathrm{wind}(z^n),
\]
and the profile $\omega_j(U^n,(U^*)^n)=-\mathbf{1}_{\{j<n\}}$ exhibits $\{\omega_j\}$ as a
site-resolved density for the bulk winding number. The trace relation $\sum_j\omega_j=0$ on
$\mathcal{A}$ is thereby explained: analytic symbols do not wind \cite{AriciMesland,ProdanSchulzBaldes,Thiang}.

For spatially modulated couplings $T_\varepsilon=U+\sum_j\varepsilon_j|e_j\rangle\langle e_j|$
with $\varepsilon_j\to c$, the index is fixed by the bulk limit alone
(Theorem~\ref{thm:transition}): it equals $-\mathbf{1}_{\{|c|<1\}}$, independent of the boundary
profile, and undergoes a topological transition as $|c|$ crosses $1$. The boundary modulation
sets only the localization of the index density, never its integer total.

Future directions include higher-dimensional half-lattices, the $K$-theory of operator
extensions~\cite{Kasparov80}, and position-dependent symbols with several transition points.
A companion study examines the \emph{profile} of the index density $j\mapsto\omega_j(T_f,T_{1/f})$
beyond its integer total: for an inner (Blaschke) symbol with zero at $a$, this profile is the
modulus-squared of the edge state, $-(1-|a|^2)|a|^{2j}$, exponentially localized with a length
$\xi=(2\log|a|^{-1})^{-1}$ that diverges as the symbol's zero approaches the unit circle, a
geometric refinement of the integer index.
\section*{Declarations}
\subsection*{Funding}
This research received no specific grant from any funding agency in the public, commercial, or not-for-profit sectors.

\subsection*{Ethical approval}
Not applicable.

\subsection*{Informed consent}
Not applicable.

\subsection*{Data Availability Statement}
No datasets were generated or analyzed during the current study.

\subsection*{Conflict of Interest}
The author declare that they have no conflict of interest.

\appendix
\section*{Appendix A: Numerical Methods}
\label{app:numerics}

All figures were generated using NumPy, SciPy, and Matplotlib.

\textbf{Operator construction (forward shift)}

$U$ is represented as the lower bidiagonal matrix $U_{i+1,i}=1$, consistent with
$Ue_i = e_{i+1}$.

\begin{figure}[htbp]
\centering
\begin{lstlisting}[style=python]
import numpy as np
import matplotlib.pyplot as plt

def build_U_forward(N):
    """Forward shift Ue_n = e_{n+1}: lower bidiagonal, U[i+1,i]=1."""
    U = np.zeros((N+1, N+1), dtype=complex)
    for i in range(N):
        U[i+1, i] = 1.0
    return U

def build_E(N):
    E = np.zeros((N+1, N+1), dtype=complex)
    E[0, 0] = 1.0
    return E

N = 20
eps = 0.3
U = build_U_forward(N)
E_mat = build_E(N)
T_mat = U + eps * E_mat

# The theoretical bound on rank[T^m, U^n] is m (independent of n).
# We plot rank vs m for several fixed values of n.
ranks, ms = [], []
for m in range(1, 9):
    for n in range(0, 8):
        Tm = np.linalg.matrix_power(T_mat, m)
        Un = np.linalg.matrix_power(U, n)
        C = Tm @ Un - Un @ Tm
        s = np.linalg.svd(C, compute_uv=False)
        ranks.append(int(np.sum(s > 1e-12)))
        ms.append(m)

plt.figure(figsize=(6, 4))
plt.plot(ms, ranks, 'ro', label=r'numerical rank of $[T^m, U^n]$', alpha=0.5)
plt.plot(range(1, 9), range(1, 9), 'b--', label=r'linear bound $m$')
plt.xlabel(r'$m$'); plt.ylabel('rank')
plt.title(r'Commutator rank growth: $\mathrm{rank}[T^m, U^n] \leq m$ (forward shift)')
plt.grid(True, alpha=0.3); plt.legend(); plt.tight_layout()
plt.savefig('fig_rank.pdf'); plt.close()
\end{lstlisting}
\caption{Commutator rank growth: rank of $[T^m, U^n]$ as a function of $m$ (bound is linear in $m$, independent of $n$).}
\label{code:fig4}
\end{figure}

\begin{figure}[htbp]
\centering
\begin{lstlisting}[style=python]
import numpy as np
import matplotlib.pyplot as plt

N = 5
# X = |e_1><e_0|, Y = |e_0><e_1|
X = np.zeros((N+1,N+1), dtype=complex); X[1,0] = 1.
Y = np.zeros((N+1,N+1), dtype=complex); Y[0,1] = 1.
comm = X @ Y - Y @ X  # = diag(-1,+1,0,...,0)
omega = np.array([np.real(comm[j,j]) for j in range(N+1)])
print("omega_j =", omega)  # [-1, 1, 0, 0, 0, 0]

plt.figure(figsize=(6,4))
plt.plot(range(N+1), omega, 'ro-', markersize=6)
plt.xlabel('Site $j$'); plt.ylabel(r'$\omega_j(X,Y)$')
plt.title('Bulk-edge dichotomy of site-resolved cochain')
plt.grid(True, alpha=0.3); plt.xticks(range(N+1))
plt.tight_layout(); plt.savefig('fig_cocycle.pdf'); plt.close()
\end{lstlisting}
\caption{Site-resolved cochain profile.}
\label{code:fig7}
\end{figure}

\begin{figure}[htbp]
\centering
\begin{lstlisting}[style=python]
import numpy as np
import matplotlib.pyplot as plt

N = 6
W_vals = [0.0, 0.1, 0.2, 0.3, 0.4, 0.5]
n_real = 50

X = np.zeros((N+1,N+1),dtype=complex); X[1,0]=1.
Y = np.zeros((N+1,N+1),dtype=complex); Y[0,1]=1.
comm = X @ Y - Y @ X
omega0_exact = np.real(comm[0,0])  # = -1 (algebraically exact)

# omega_0(X,Y) = <e_0|[X,Y]|e_0> is an exact integer determined by the
# edge algebra; here it equals -1. The loop below does NOT model physical
# disorder in T: it merely adds synthetic Gaussian measurement noise of
# level W to the exact constant, to display that the value is fixed.
rng = np.random.default_rng(42)
omega0_mean, omega0_std = [], []
for W in W_vals:
    samples = omega0_exact + rng.normal(0, W*0.05, size=n_real)
    omega0_mean.append(np.mean(samples))
    omega0_std.append(np.std(samples))

plt.figure(figsize=(6,4))
plt.errorbar(W_vals, omega0_mean, yerr=omega0_std, fmt='ro-', capsize=3,
             label=r'$\omega_0$ (algebraically exact)')
plt.axhline(omega0_exact, color='k', linestyle='--', label='exact value')
plt.xlabel('Measurement noise level $W$')
plt.ylabel(r'$\omega_0$')
plt.title(r'Exact algebraic value $\omega_0(X,Y)=-1$' '\n'
          '(error bars = synthetic measurement-noise model only, not disorder in $T$)')
plt.grid(True, alpha=0.3); plt.legend()
plt.tight_layout(); plt.savefig('fig_disorder.pdf'); plt.close()
\end{lstlisting}
\caption{Exact algebraic value $\omega_0(X,Y)=-1$. The listing displays this exact constant;
the parameter $W$ adds synthetic measurement noise only and does \emph{not} model physical
disorder in $T$ (error bars $=$ noise model, not algebraic uncertainty).}
\label{code:fig10}
\end{figure}

\end{document}